\documentclass[a4paper,10pt]{article}
\usepackage{amsmath}
\usepackage{amsfonts}
\usepackage{amssymb}
\usepackage{enumerate}
\usepackage{graphics}
\RequirePackage{geometry}
\newtheorem{theorem}{Theorem}
\newtheorem{lemma}[theorem]{Lemma}
\newtheorem{definition}[theorem]{Definition}

\newcommand{\prob}{\mathbb{P}}
\newcommand{\qed}{\hfill $\square$}

\newcommand{\xn}{\lfloor xn\rfloor}
\newcommand{\exn}{\lfloor \varepsilon \xn\rfloor}

\def\lra{\leftrightarrow}
\def \Z {\mathbb Z}

\def \N {\mathbb N}
\def \E {\mathbb E}
\def \P {\mathbb P}

\def\al{\alpha}

\def\be{\beta}
\def\ep{\varepsilon}
\def\de{\delta}

\def\La{\Lambda}
\def\ga{\gamma}

\def\de{{\delta}}

\begin{document}
\title{On the size of the largest cluster in $2D$ critical percolation}
\author{J. van den Berg\footnote{CWI and VU University Amsterdam}
\, and  R. Conijn\footnote{VU University Amsterdam} \\
{\footnotesize email: J.van.den.Berg@cwi.nl, R.P.Conijn@vu.nl}
}
\date{}

\maketitle

\begin{abstract}
We consider (near-)critical percolation on the square lattice.
Let $\mathcal{M}_{n}$ be the size of the largest open cluster contained in the box $[-n,n]^2$,
and let $\pi(n)$ be the probability that there is an open path from $O$ to the boundary of the box.
It is well-known (see \cite{BCKS01}) that for all $0< a < b$ the probability that $\mathcal{M}_{n}$ is smaller
than $a n^2 \pi(n)$ and the probability that $\mathcal{M}_{n}$ is larger than $b n^2 \pi(n)$ are
bounded away from $0$ as $n \rightarrow \infty$. 
It is a natural question, which arises for instance in the study of so-called frozen-percolation processes,
if a similar result holds for the probability that $\mathcal{M}_{n}$
is {\em between} $a n^2 \pi(n)$ and $b n^2 \pi(n)$. 
By a suitable partition of the box, and a careful construction involving the building blocks, we show that the
answer to this question is affirmative. 
The `sublinearity' of $1/\pi(n)$ appears to be essential for the argument.
\end{abstract}

\section{Introduction and main result}\label{sec:introduction}
Consider bond percolation on $\mathbb{Z}^{2}$ with parameter $p$.
(See \cite{G99} for a general introduction to percolation theory.) Let $\Lambda_{n} = [-n,n]^2 \cap \mathbb{Z}^2$
and let, for $v \in \Lambda_{n}$,  $\mathcal{C}_{n}(v)$ denote the size of the open cluster of $v$ inside the box $\Lambda_{n}$:
\begin{equation*}
 \mathcal{C}_{n}(v) = |\{ w \in \Lambda_{n}: v \lra w, \quad \textrm{inside}\; \Lambda_{n}\}|,
\end{equation*}
where we use the standard notation $v \lra w$ for the existence of an open path from $v$ to $w$, and
where the addition `inside $\La_n$' means that we require the existence of such a path which is located entirely in
$\La_n$. 
For a set $A \subset \mathbb{Z}^{2}$ we denote by $\partial A$ the (internal) boundary of $A$:
\begin{equation*}
 \partial A = \{ v \in A: \exists w \not\in A: (v,w)\textrm{ is an edge } \}.
\end{equation*}
The remaining part of $A$ will be called the interior of $A$.
Let $\pi_{p}(n)$ be the probability $\prob_{p}(O \lra \partial \Lambda_{n})$.
For simplicity we write $\pi(n)$ for $\pi_{\frac{1}{2}}(n)$.

We are interested in the size of `large' open clusters in $\La_n$ for the case where $p$ is equal (or close) to the
critical value $1/2$.
It is known in the literature that, informally speaking, the size of the largest open cluster is typically
of order $n^2 \pi(n)$: For any $c > 0$, there is a `reasonable' probability that it is larger (smaller)
than $c n^2 \pi(n)$, and this probability goes to $0$ uniformly in $n$ as $c \rightarrow \infty$
($c \rightarrow 0$). (See \cite{BCKS99}, \cite{BCKS01}; see also \cite{J03} Section 3.)
However, the question whether for all $0 < a < b$ there is a `reasonable' probability that there is an open
cluster with size {\em between} $a n^2 \pi(n)$ and
~$b n^2 \pi(n)$ has not been investigated in the literature.
 
This question, which is also natural by itself, arises e.g. in the study
of finite-parameter frozen-percolation models.
In these models each edge is closed at time $0$ and `tries' to become open
at some random time, independently of the other edges. However, an open cluster stops growing as soon as its
size has reached a certain
(large) value $M$, the parameter of the model. (See \cite{BLN12} where this was studied for the case where
the `size' of a cluster is defined as its diameter instead of its volume.)
The investigation of such processes leads to the question
how two open clusters which both have size of order $M$ but smaller than $M$, merge to a cluster of size bigger than $M$,
which in turn leads to the question at the end of the previous paragraph.
To state our main result, an affirmative answer to that question, we first need a few more definitions.

For $k,l \in \mathbb{N}$, we denote by $HC(k,l)$ the event that there is an open horizontal crossing in the box
$[0,k]\times [0,l]$. (This is an open path from the left side to the right side of the box, of which all vertices, except
the starting and end point, are in the interior of the box). Let the ``characteristic length'' be as defined in e.g.
\cite{N08} and \cite{K87}:
For a fixed  $\epsilon \in (0,\frac{1}{2})$:
\begin{equation}
 L(p) = L_{\epsilon}(p) = \left\{ \begin{array}{ll} \min{\{ n\in \mathbb{N}: \prob_{p}(HC(n,n)) \le \epsilon\} } & \textrm{if } p < \frac{1}{2},\\
                          \min{\{ n \in \mathbb{N}: \prob_{p}(HC(n,n)) > 1 - \epsilon\} } & \textrm{if } p >\frac{1}{2},
                         \end{array}\right.
\end{equation}
and $L(\frac{1}{2}) = \infty$.
The precise value of $\epsilon$ is not essential. Throughout this paper we will consider it as being fixed, and therefore
we omit it from our notation.

As said before, our main question concerns the existence of {\em some} open cluster in $\La_n$ with size in some
specific interval.
The proof we obtained gives, with only a tiny bit of extra work, something stronger;
it shows that 
with `reasonable' probability the {\em maximal} open cluster has this property. 
Therefore we state our main result in this
stronger form (and remark that we do not know an essentially simpler proof of the original weaker form): \\
Denote by $\mathcal{M}_{n}$ the size of the maximal open cluster in $\Lambda_{n}$. More precisely,
\begin{equation*}
\mathcal{M}_{n} = \max_{v \in \Lambda_{n}}{ \mathcal{C}_{n}(v) }.
\end{equation*}

\begin{theorem}\label{thm:maxClusterInInt}
 Let $0 < a < b$. There exist $\delta > 0$ and $N \in \mathbb{N}$ such that, for all $n \geq N$ and all $p$ with
$L(p) \geq n$,
\begin{equation*}
 \prob_{p}\left(\mathcal{M}_{n} \in (an^2 \pi(n), bn^2 \pi(n))\right) > \delta.
\end{equation*}
\end{theorem}
\smallskip
The proof is given in Section \ref{sect-proof}. Section \ref{sect-notation} will list the main ingredients 
used in the
proof. The proof involves a suitable partition of $\La_n$ in smaller boxes and annuli. A brief and informal summary
is given in the beginning of Section \ref{sect-proof}, after the description of these objects. 

\section{Ingredients for the proof of Theorem \ref{thm:maxClusterInInt}}\label{sect-notation}
We will make amply use of standard RSW results of the following form: 
For all $l > 0$ there exists $\delta(l) > 0$ such that, for all $k$ and all $p$ with
$L(p) \geq k$,
$\prob_{p}(HC(k,\lceil l k \rceil)) \ge \delta(l)$.
%
%
For a set $W$ of vertices define 
\begin{equation}\label{eq:defCtilde}
 \tilde{\mathcal{C}}(W) = |\{ v \in W: v \lra \partial W\}|.
\end{equation}
For $k, n \in \N$, we use the notation $\La_{k,n}$ for the rectangle $[-k,k] \times [-n,n]$.
We will use the following properties of $\pi_{p}(n)$ from the literature. 
\begin{theorem}\label{lem:PiBounds} There exist $\alpha, C_{1}, \cdots, C_6 >0$ such that: 
 \begin{enumerate}[(i)]
  \item For all $m \le n$:
   \begin{equation*}
    C_{1}(\frac{n}{m})^{\alpha} \le \frac{\pi(m)}{\pi(n)} \le  C_{2}(\frac{n}{m})^{\frac{1}{2}}.
   \end{equation*}
  \item For all $n \in \mathbb{N}$:
   \begin{equation*}
    \sum_{k=0}^{n} \pi(k) \le C_{3}\cdot n\pi(n).
   \end{equation*}
  \item For all $p\in (0,1)$ and all $n \le L(p)$,
   \begin{equation*}
    C_{4}\pi(n) \le \pi_{p}(n) \le C_{5}\pi(n).
   \end{equation*}
   \item 
For all $k, n \in \mathbb{N}$ and $p$ with $L(p) \ge k \wedge n$, 
   \begin{equation*}
   \mathbb{E}_p[\tilde{\mathcal{C}}(\Lambda_{k,n})] \leq C_6 k n \, \pi( k\wedge n).
   \end{equation*}
 \end{enumerate}
\end{theorem}
\begin{proof}
The inequalities in (i) are well-known. (The first follows easily from RSW arguments, and the second 
goes back to \cite{BK85}; see also for example \cite{BCKS01}).
Part (ii) follows from (7) in \cite{K86}.
Part (iii) is Theorem 1 in \cite{K87}.
Part (iv), of which versions  are explicitly in the literature (see e.g. \cite{K86}, \cite{K87} and \cite{N08}),
is proved as follows (where we assume that $k \leq n$):
\begin{eqnarray}
 \, & \, &  \mathbb{E}_p[\tilde{\mathcal{C}}(\Lambda_{k,n})] = \sum_{v \in \Lambda_{k,n}}
\prob_p(v \lra \partial \Lambda_{k,n}) \nonumber\\
\,  & \le & \sum_{v \in \Lambda_{k,n}} \pi_p(d(v, \partial \Lambda_{k,n})) \nonumber
\,\, \le \,\, 8n\sum_{l=0}^{k} \pi_{p}(l) \le C_{6} n k \pi(k),
\end{eqnarray}
where the last inequality uses part (ii) and (iii). \qed
\end{proof}


\medskip
Define
\begin{equation*}
 Y(m) = |\{ v \in \Lambda_{m}: v \lra \partial\Lambda_{2m} \}|.
\end{equation*}
We need the following result for the distribution of $Y(m)$, which is essentially in
\cite{BCKS99} and (for the special case $p = 1/2$) \cite{K86}.

\begin{theorem} \label{lem:momentBoundExpConnBoundary}
There exist $\delta_{1}, C_{7} > 0$ such that, for all $p\in (0,1)$ and all $m \le L(p)$:
\begin{equation}\label{eq:posProbEnoughVertConn}
 \prob_{p}(Y(m) \ge C_{7}m^{2}\pi(m)) \ge \delta_{1}.
\end{equation}
\end{theorem}
\begin{proof}
By Lemma 6.1 in \cite{BCKS99} there
exists $C_{8} > 0$ such that for all $p \in (0,1)$ and $m \le L(p)$,
$\mathbb{E}_{p}[(Y(m))^{2}] \le C_{8} (m^{2}\pi(m))^{2}$.
Further, by the definition of $\pi(n)$ and parts (i) and (iii) of Theorem \ref{lem:PiBounds},
there exists $C_{9} > 0$ such that
$\mathbb{E}_{p}[Y(m)] \ge C_{9}m^{2}\pi(m)$.
These two inequalities, and the one-sided Chebyshev's inequality, give
Theorem \ref{lem:momentBoundExpConnBoundary}. \qed
\end{proof}

\medskip
It was shown in \cite{K86} (and extended/generalized in \cite{BCKS99}
and \cite{BCKS01}) that $\mathcal{M}_{n}$, the size of the largest open cluster in $\La_n$, is typically of
order $n^2 \pi(n)$. In particular, its expectation has an upper and a lower bound which are linear in
$n^2 \pi(n)$. In the proof of Theorem \ref{thm:maxClusterInInt} we use the following result from \cite{BCKS01}.
\begin{theorem}\label{thm:maxClustCanBeSmall} {\em \bf (\cite{BCKS01} Thm. 3.1 (i), Thm. 3.3 (ii))} \\
Let $p_{n}$ be a sequence, such that $n \le L(p_{n})$ for all $n$. Then 
for all $K > 0$,
 \begin{equation*}
  \liminf_{n \to \infty} \prob_{p_{n}}\left(\frac{\mathcal{M}_{n}}{n^2 \pi(n)} < K\right) > 0.
 \end{equation*}
\end{theorem}

Finally, to streamline the arguments in Section \ref{sect-compl} at the end of the proof of Theorem \ref{thm:maxClusterInInt},
we state here the following fact about `steering' the outcome of the sum of independent random variables.
It is a simple observation rather than a lemma, and versions of it have without doubt been
used in the probability literature in various contexts.
\begin{lemma}\label{lem:sumOfRandomVars} Let $0 < \al < \be$,
and let $k \in \N$ be such that $\al/k < (\be - \al)/2$.
Further, let $\eta_{1}, \eta_{2} > 0$ and let $(X_{i})_{1 \leq i \le k}$ be independent random variables,
(not necesarrily identically distributed) which satisfy the following:
\begin{eqnarray*}
 \prob\left(X_{i} \in (\frac{\al}{k}, \frac{\be-\al}{2})\right) & \ge & \eta_{1};\\
 \prob\left(X_{i} \le \frac{\be-\al}{2k} \right) & \ge & \eta_{2}.
\end{eqnarray*}
Then
\begin{equation*}
 \prob\left(\sum_{i=1}^{k} X_{i} \in (\al, \be)\right) \ge (\eta_{1}\wedge \eta_{2})^{k}.
\end{equation*}
\end{lemma}
\begin{proof}
For $1 \leq i \leq k$ we say that `step $i$ is proper' if
$$
X_i  \begin{cases} \in (\frac{\al}{k}, \frac{\be - \al}{2}) & \mbox{ if } \sum_{j=1}^{i-1} X_j < \al \\
\leq \frac{\be - \al}{2 k} & \mbox{ otherwise. } \end{cases}
$$ 
It is clear that if all steps $i =1, \cdots, k$ are proper, then $\sum_{i=1}^k X_i \in (\al,\be)$.
It is also easy to see that, for each $i$, the
conditional probability that step $i$ is proper, given that all steps $1, \cdots, i-1$ are proper,
is at least $\min(\eta_1, \eta_2$). \qed
\end{proof}

\section{Proof of Theorem \ref{thm:maxClusterInInt}} \label{sect-proof}
We first give a proof for the special case $p = 1/2$ and therefore drop the subscript $p$ from the notation $\P_p$ and
$\E_p$. At the end of Section \ref{sect-compl} we point out that (due to the `uniformity' of the ingredients stated
in Section \ref{sect-notation}) the proof for the general case is essentially the same.

\subsection{More definitions, and brief outline of the proof}\label{sec:proofDefi}

Let $s, t \in \mathbb{N}$ with $t \le \frac{1}{3}s$.
The proof involves a construction using the following boxes and annuli.
\begin{eqnarray*}
 B_{0,0} & = & \Lambda_{s}.\\
 A_{0,0}^{I} & = & \Lambda_{s} \setminus \Lambda_{s - t}; \,\,\,\,
A_{0,0}^{II}  =  \Lambda_{s -t} \setminus \Lambda_{s - 2t}; \,\,\,\,
 A_{0,0}^{III}  =  \Lambda_{s -t} \setminus \Lambda_{s - 2t}. \\
 A_{0,0}' & = & A_{0,0}^{I} \cup A_{0,0}^{II} \cup A_{0,0}^{III}.\\
 H_{0,0} & = & ([0,4t]\times [0,t] + (s-2t, 0)) \cap \mathbb{Z}^{2}.\\
 V_{0,0} & = & ([0,t]\times [0,4t] + (0, s-2t)) \cap \mathbb{Z}^{2}.
\end{eqnarray*}
More generally, for all $i,j \in \mathbb{Z}$ we define $B_{i,j} = B_{0,0} + (2is,2js)$, 
$A_{i,j}^{I} = A_{0,0}^{I} + (2is,2js)$, etcetera.

\begin{figure}
 \centering
 \scalebox{1.5}{\includegraphics{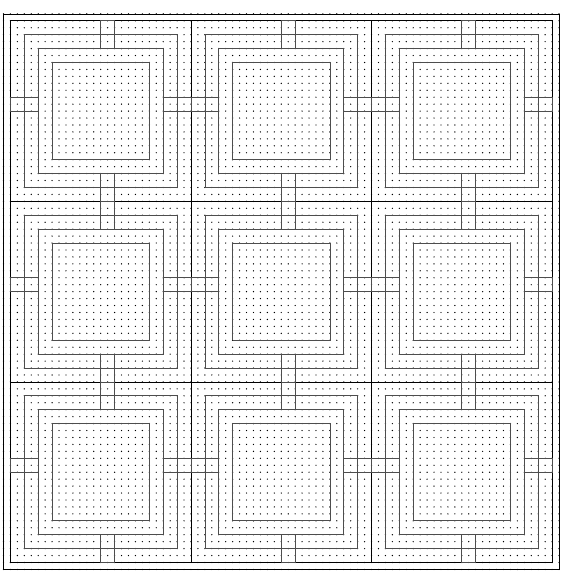}}
 \caption{Partition of the box $\Lambda_{n}$. Here $n=40, m = 3, s=13, t=2$.}
 \label{fig:smallboxes}
\end{figure}
\smallskip
Before we go on, we give a very brief and informal summary of the proof (see Figure \ref{fig:smallboxes}):
The box $\La_n$ in the statement of the theorem will
be (roughly) partitioned in $m^2$ boxes $B_{i,j}$ defined above, where the $s$ (and hence $m$) and $t$ will be
chosen appropriately, depending on $n$, $a$ and $b$. (For elegance/symmetry we take the number $m$ odd).
We will `construct' an open cluster of which the `skeleton' consists of circuits in the annuli $A_{i,j}^{II}$,
`glued' together by connections in the `corridors' $V_{i,j}$ and $H_{i,j}$. (The other annuli defined
above will be used for technical reasons in the proof). The setup is such that
the contributions from the different $B_{i,j}$'s to the total cluster size are roughly independent,
and that these contributions can be `steered' to get the total sum inside the desired interval.
In some sense this replaces the original problem for
the box $\La_n$ by a similar problem, but now for the smaller boxes $B_{i,j}$.
Apart from the technicalities involving the
control of local dependencies, there is a subtle aspect in the proof related to the
asymptotic behaviour of $\pi(n)$: Although the precise power-law behaviour of $\pi(n)$ is not
important, it seems to be essential for the arguments that the exponent in a power-law upper
bound is strictly smaller than $1$ (see the note at the
end of the proof of Lemma \ref{lem:chooseNxEps})).

\medskip
Now we continue with the precise constructions mentioned above. First we give some more notation and definitions.
Let $E$ denote the set of edges of $\mathbb{Z}^{2}$ and $\Omega = \{0,1\}^{E}$.
For $\omega \in \Omega$ and $F \subset E$ we will write $\omega_{F} \in \{0,1\}^{F}$ for the `restriction'
$(\omega_e, e \in F)$ of $\omega$ to $F$.
Let $A \subset \Omega$ and $W \subset \mathbb{Z}^{2}$. We write $E(W)$ for the set of all edges of
which both endpoints are in $W$.
Informally, we use the notation $A(W)$ for the set of all configurations $\omega \in \Omega$ that belong to $A$
or can be turned to an element of $A$ by modifying $\omega$ outside $E(W)$.
More precisely,
\begin{equation}\label{eq:notationRestrOfIncrEvent}
 A(W) = \{\omega_{E(W)}: \omega \in A\} \times \{ 0,1\}^{E\setminus E(W)}.
\end{equation}

\begin{figure}
 \centering
 \scalebox{0.4}{\includegraphics{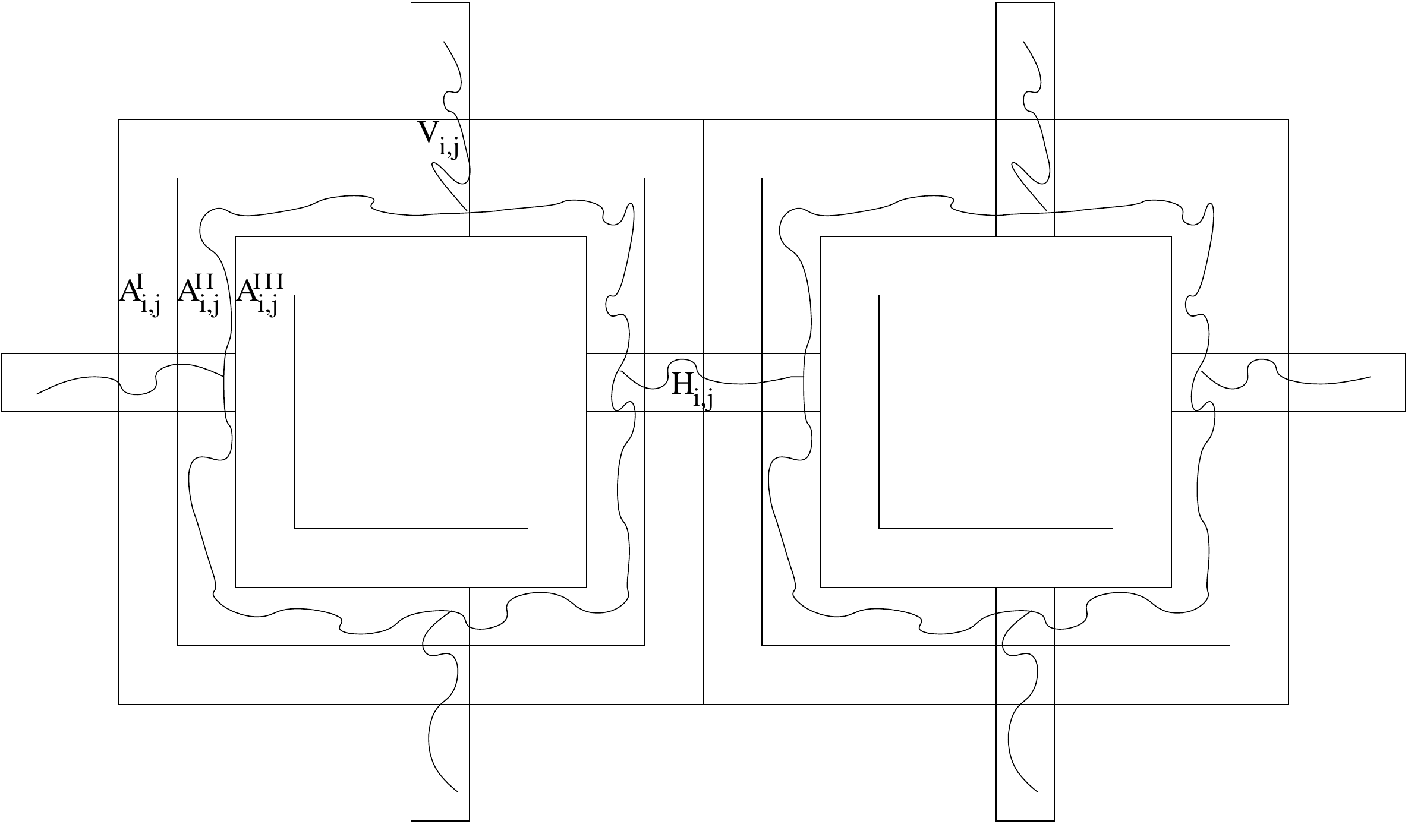}}
 \caption{Illustration of the event $\tilde{O}^{s,t}$.}
 \label{fig:eventO}
\end{figure}
We denote by $\tilde{O}^{s,t}$ the event that (i) - (iii) below occur (see Figure \ref{fig:eventO}):
\begin{enumerate}[(i)]
 \item $\forall i,j\in \mathbb{Z}$: the annulus $A_{i,j}^{II}$ contains an open circuit;
 \item $\forall i,j\in \mathbb{Z}$: $H_{i,j}$ contains an open connection between the two widest open circuits in the annuli $A_{i,j}^{II}$
and $A_{i+1,j}^{II}$;
 \item  $\forall i,j\in \mathbb{Z}:$ $V_{i,j}$ contains an open connection between the two widest open circuits in the annuli $A_{i,j}^{II}$
and $A_{i,j+1}^{II}$.
\end{enumerate}
The introduction of this event looks meaningless since it has probability $0$. It will only be used
to give a `compact' description of the following events (which do play a key role in the proof). 
\begin{definition}\label{def:O} Let $m, s, t \in \mathbb{N}$, with $t \le \frac{1}{3}s$ and $m$ odd. Let $i, j \in \Z$.
We define, using notation \eqref{eq:notationRestrOfIncrEvent}, the following events:
 \begin{eqnarray*}
  O^{m,s,t} & = & \tilde{O}^{s,t}(\Lambda_{ms}).\\
  O_{i,j}^{s,t} & = & \tilde{O}^{s,t}(B_{i,j}).
 \end{eqnarray*}
\end{definition}
\textbf{Remark:} From now on, for given $m,s,t$, the indices $i,j$ under consideration will always be assumed to
be in the set $\{-\frac{1}{2}(m-1), \cdots , 0, \cdots, \frac{1}{2}(m-1)\}$.

\subsection{Expected cluster size in a narrow annulus}\label{sect-czann}
For a circuit $\gamma$ in $\mathbb{Z}^{2}$ we denote by Int$(\gamma)$
the bounded connected component of $\mathbb{Z}^{2} \setminus \gamma$, and define
\begin{equation}\label{eq:defCijGamma}
\mathcal{C}^{\gamma} = |\{ v \in \mathrm{Int}(\gamma): v \lra \gamma \}|.
\end{equation}
Further, for all $i, j$, let $\gamma_{i,j}$ denote the widest open circuit in the annulus $A_{i,j}^{II}$, and
define, for $W \subset \La_n$,
\begin{equation} \label{eq:defCij}
 \mathcal{C}_{i,j}(W) =  |\{ v \in W: v \lra \gamma_{i,j} \}|.
\end{equation}
If there is no  open circuit in $A_{i,j}^{II}$, then $\mathcal{C}_{i,j}(W) = 0$.

\smallskip
Recall the definition of $\tilde{\mathcal{C}}$ in \eqref{eq:defCtilde}.
\begin{lemma}\label{lem:ExpSizeEpsxX} There exists a constant $C_{10} > 0 $ such that for all
$s \in \mathbb{N}$, $t \le \frac{1}{3}s$ and all $i,j$:
\begin{equation*}
 \mathbb{E}[\mathcal{C}_{i,j}(A_{i,j}') | O_{i,j}^{s,t}] \le \mathbb{E}[\tilde{\mathcal{C}}(A_{i,j}') | O_{i,j}^{s,t}]
\le C_{10} st \pi(t).
\end{equation*}
\end{lemma}
\begin{proof}
The first inequality follows immediately from the obvious fact that,
on the event $O_{i,j}^{s,t}$, $\mathcal{C}_{i,j}(A_{i,j}')$
is smaller than or equal to $\tilde{\mathcal{C}}(A_{i,j}')$. We prove the second inequality.
Without loss of generality we take $i=j=0$.
\begin{figure}
 \centering
 \scalebox{0.3}{\includegraphics{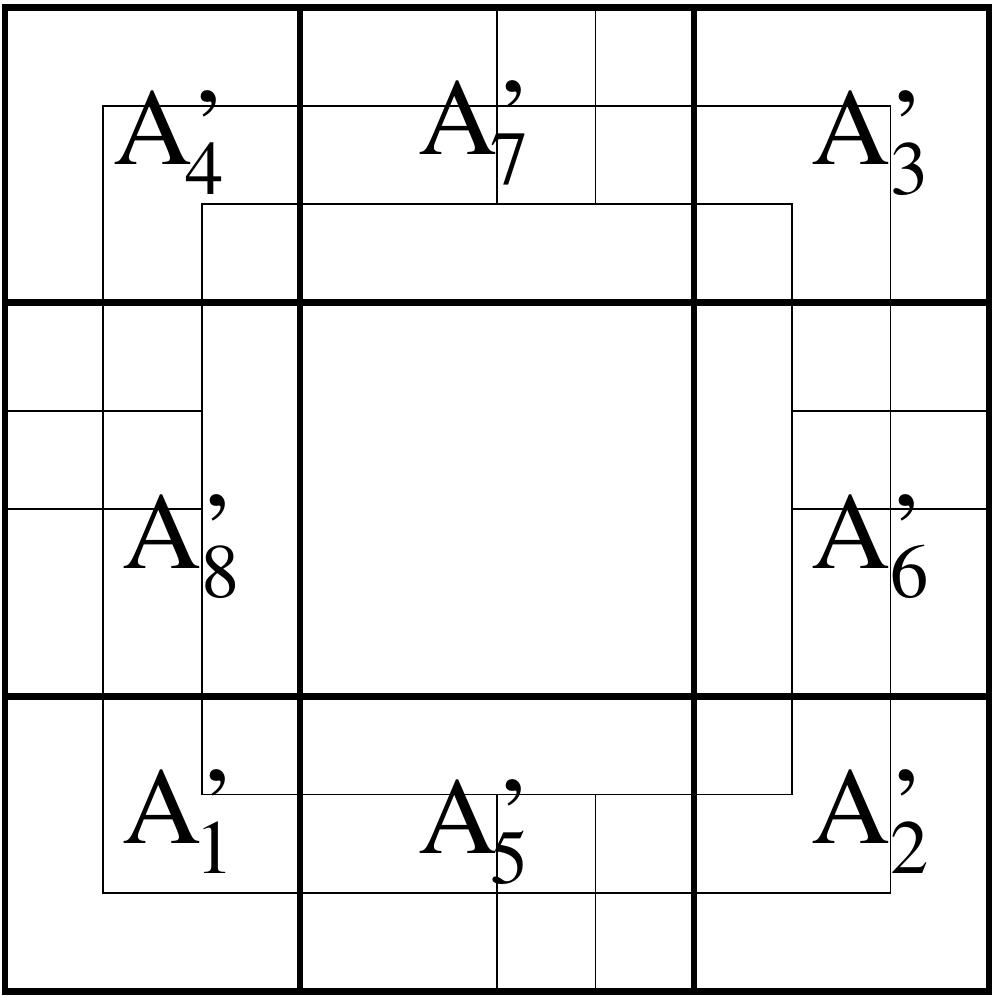}}
 \caption{The subdivision of $A_{0,0}'$ in $A_{1}', \cdots, A_{8}'$.}
 \label{fig:PartAPri}
\end{figure}
We subdivide the annulus
$A' = A_{0,0}' = \bigcup_{l=1}^{8} A_{l}'$, where $A_{1}', A_{2}', A_{3}', A_{4}'$ are the $(3t \times 3t)$-squares in
the four corners,
and $A_{5}', A_{6}', A_{7}', A_{8}'$ the remaining rectangles (see Figure \ref{fig:PartAPri}).
Note that
$\tilde{\mathcal{C}}(A') \le \sum_{l=1}^{8} \tilde{\mathcal{C}}(A_{l}').$
Hence it is sufficient to show that there is a constant $C_{11}$ such that for each $l =1, \cdots,8$,
\begin{equation} \label{eq-r-tCbd}
\E[\tilde{\mathcal{C}}(A_{l}') | O_{0,0}^{s,t}] \leq C_{11} s t \pi(t).
\end{equation}
By symmetry we only have to handle the cases $l=1$ and $l=5$.
For each $l$ the l.h.s. of \eqref{eq-r-tCbd} is
\begin{equation} \label{eq-r-help1}
\mathbb{E}[\tilde{\mathcal{C}}(A_{l}') | O_{0,0}^{s,t}] = \frac{1}{\prob(O_{0,0}^{s,t})}\sum_{v \in A_{l}'}
\prob(v \lra \partial A_{l}'; O_{0,0}^{s,t}).
\end{equation}
Recall the notation \eqref{eq:notationRestrOfIncrEvent}. For each $v \in A_{1}'$, obviously,
\begin{equation} \label{eq-r-help2}
 \prob(v \lra \partial A_{1}'; O_{0,0}^{s,t}) \le
\prob(v \lra \partial A'_1) \prob(O_{0,0}^{s,t}(A' \setminus A_{1}')).
\end{equation}
Further, informally speaking, the event $O_{0,0}^{s,t}(A' \setminus A_{1}')$ can, with a `local surgery
involving a bounded cost in terms of probability', be turned into the event $O_{0,0}^{s,t}$.
More precisely, if $O_{0,0}^{s,t}(A' \setminus A_{1}')$ holds, and there is a horizontal open crossing
of the rectangle $[-s, -s + 6t] \times [-s + t, -s + 2t]$ and of the square $[-s, -s + 3 t] \times 
[-s + 3t, -s + 6 t]$, and a vertical open crossing of the rectangle
$[-s +t, -s + 2t] \times [-s, -s + 6t]$ and of the square $[-s + 3 t, -s + 6 t] \times [-s, -s + 3 t]$, then
the event $O_{0,0}^{s,t}$ holds. Hence, by RSW (and FKG) we have a positive constant $C_{12}$ such that
$\prob(O_{0,0}^{s,t}(A'\setminus A_{1}')) \le C_{12} \prob(O_{0,0}^{s,t})$.
Combining this with \eqref{eq-r-help1} and \eqref{eq-r-help2} gives
\begin{equation}\label{eq-r-case1}
\mathbb{E}[\tilde{\mathcal{C}}(A_{1}') | O_{0,0}^{s,t}]  \le C_{12} \sum_{v \in A'_1} \prob(v \lra \partial A'_1).
\end{equation}
For the case $l=5$ a similar argument (now the `surgery' can be done on the union of the region $V_{0,-1}$ and the rectangle
$[v_{1} - t, v_{1} + t]\times [-s,-s+3t]$)) gives a constant $C_{13}> 0$ such that
\begin{equation}\label{eq-r-case5}
\mathbb{E}[\tilde{\mathcal{C}}(A_{5}') | O_{0,0}^{s,t}]  \le 
C_{13} \sum_{v \in A'_5}  \prob(v \lra \partial A'_5).
\end{equation}
Application of part (iv) of Theorem \ref{lem:PiBounds} to the right-hand sides of \eqref{eq-r-case1} and
\eqref{eq-r-case5} gives \eqref{eq-r-tCbd}. \qed
\end{proof}

\subsection{Properties of nice circuits}
Let $m,s,t$ be as in Definition \ref{def:O}, and recall the Remark about the values of the indices
$i, j$ at the end of Section \ref{sec:proofDefi}. Let, for each $i, j$, $\ga_{i,j}$ be as in the beginning
of Section \ref{sect-czann},
and let $\tilde{\ga}_{i,j}$ be a deterministic circuit in the annulus $A_{i,j}^{II}$.
Further we will denote the collection of all $\ga_{i,j}$'s by $(\ga)$, and the collection
of al $\tilde{\gamma}_{i,j}$'s by $(\tilde \gamma)$.
\begin{definition}\label{def:nice}
 We say that $\tilde{\gamma}_{i,j}$ is \emph{$(s,t)$-nice} if
\begin{equation}\label{eq:defGammaij}
 \mathbb{E}[\tilde{\mathcal{C}}(A_{i,j}')\: | \; O_{i,j}^{s,t}; \gamma_{i,j} = \tilde{\gamma}_{i,j}] \le 2 C_{10} s t \pi(t),
\end{equation}
with $C_{10}$ as in Lemma \ref{lem:ExpSizeEpsxX}.
Further, the collection $(\tilde{\gamma})$ is called $(m,s,t)$-nice if 
each circuit in the collection is $(s,t)$-nice.
\end{definition}
We define $\Gamma_{i,j}^{s,t}$
as the event that $\ga_{i,j}$ is $(s,t)$-nice, and define

\begin{equation}\label{eq:defGammaNxEps}
 \Gamma^{m,s,t} = \bigcap_{i,j} \Gamma_{i,j}^{s,t}.
\end{equation}
Recall the notation \eqref{eq:defCij}. We will study the open cluster $\mathcal{C}_{0,0}(Q)$, where
\begin{equation*}
 Q = (\Lambda_{n} \setminus \Lambda_{ms}) \cup (\bigcup_{i,j} A_{i,j}').
\end{equation*}
\begin{lemma}\label{lem:probGamma}
There exist positive constants $C_{14}$ and $C_{15}$ such that, for all $m, s$ and $t$,
\begin{enumerate}[(i)]
\item 
\begin{equation*}
 \prob(\Gamma^{m,s,t} | O^{m,s,t}) \ge (C_{14})^{m^{2}}.
\end{equation*}
\item For all $(m,s,t)$-nice $(\tilde{\gamma})$ and all $n$ with $n - ms \le t$,
\begin{equation*}
 \prob\left(\mathcal{C}_{0,0}(Q) \le C_{15} m^{2} st \pi(t)\; | \; O^{m,s,t}; (\gamma) = (\tilde{\gamma}) \right)
\ge \frac{1}{2}.
\end{equation*}
\end{enumerate}
\end{lemma}
\begin{proof} We claim that there is a constant $C_{16} > 0$ such that for all $i,j$:
\begin{equation}\label{eq:SplitOnxEpsWithGamma}
 \prob(O^{m,s,t}; (\gamma) = (\tilde{\gamma})) \ge C_{16}\prob(D)\prob(O_{i,j}^{s,t}; \gamma_{i,j} = \tilde{\gamma}_{i,j}),
\end{equation}
where (with the notation \eqref{eq:notationRestrOfIncrEvent})
\begin{equation*}
 D = O^{m,s,t}(\Lambda_{n} \setminus B_{i,j})\cap \bigcap_{\tilde{i},\tilde{j}:(\tilde{i},\tilde{j}) \neq (i,j)} \{ \gamma_{\tilde{i},\tilde{j}} = \tilde{\gamma}_{\tilde{i},\tilde{j}} \}.
\end{equation*}
To prove this claim we write
\begin{equation*}
 G = B_{i,j} \setminus A_{i,j}^{I},\qquad 
 J = A_{i,j}^{I} \cup \bigcup_{(\tilde{i},\tilde{j}) \in M_{i,j}} A_{\tilde{i},\tilde{j}}^{I},\qquad
 K = \Lambda_{n} \setminus (G \cup J),
\end{equation*}
where
$M_{i,j} =  \{(i-1,j),(i,j-1),(i+1,j),(i,j+1) \} \cap \{ -(m-1)/2, \cdots,(m-1)/2 \}^{2}$.
Let $D_{1} = O_{i,j}^{s,t} \cap \{ \gamma_{i,j} = \tilde{\gamma}_{i,j} \}$. We also need an event $D_2$ which, informally 
speaking, connects the structures in the definition of $D_1$ with those in $D$. More precisely,
$$D_{2} = \bigcap_{(\tilde{i},\tilde{j}) \in M_{i,j}} D_{2}^{\tilde{i},\tilde{j}},$$
where $D_{2}^{i+1,j}$ is the event that (i) $H_{i,j} \cap J$ contains a horizontal crossing and (ii) $H_{i,j} \cap A_{i,j}^{I}$
and $H_{i,j} \cap A_{i+1,j}^{I}$ both contain a vertical crossing. The other $D_{2}^{\tilde{i},\tilde{j}}$'s are 
defined similarly.
By RSW (and FKG) there is a positive constant $C_{16}$ such that $\prob(D_{2}) >  C_{16}$.
Note that (with the notation in Section \ref{sec:proofDefi}),
the event $D$ is increasing with respect to the edges outside $E(K)$, and that
$D_1$ is increasing with respect to the edges outside $E(G)$.
We get
\begin{eqnarray*}
 \lefteqn{\prob(O^{m,s,t}; (\gamma) = (\tilde{\gamma})) \ge \prob(D \cap D_{1} \cap D_{2})}\\
 & = & \sum_{\omega_{1} \in \{  0,1\}^{E(K)}} \prob(\omega_{1}) \sum_{\omega_{2} \in \{0,1\}^{E(G)}}
\prob(D \cap D_{1} \cap D_{2} | \omega_{1}, \omega_{2})\prob(\omega_{2})\\
 & \ge & \sum_{\omega_{1} \in \{  0,1\}^{E(K)}} \prob(\omega_{1}) \sum_{\omega_{2} \in \{0,1\}^{E(G)}} \prob(D | \omega_{1},\omega_{2})\prob(D_{1} | \omega_{1},\omega_{2})\prob(D_{2} | \omega_{1},\omega_{2})\prob(\omega_{2})\\
 & = & \sum_{\omega_{1} \in \{  0,1\}^{E(K)}} \prob(\omega_{1})\prob(D | \omega_{1})\sum_{\omega_{2} \in \{0,1\}^{E(G)}} \prob(D_{1} | \omega_{2})\prob(D_{2})\prob(\omega_{2})\\
 & \ge & C_{16} \prob(D)\prob(D_{1}),
\end{eqnarray*}
where we used FKG in the third line, and in the fourth line we used that
$D$ doesn't depend on the configuration on $G$,
$D_{1}$ doesn't depend on the configuration on $K$, and $D_{2}$ doesn't depend on the configuration on $G \cup K$.
This proves the claim.

By repeating the same arguments for each $B_{i,j}$, we eventually get the following `extension' of \eqref{eq:SplitOnxEpsWithGamma}: 


\begin{equation}\label{eq:SplitTotal}
 \prob\left((\gamma) = (\tilde{\gamma}) ; O^{m,s,t}\right) \ge
C_{16}^{m^{2}} \prod_{i,j} \prob(\gamma_{i,j} = \tilde{\gamma}_{i,j}; O_{i,j}^{s,t}).
\end{equation}

Now we are ready to prove part (i):
\begin{eqnarray}
 \prob(\Gamma^{m,s,t} | O^{m,s,t}) & = &
\frac{1}{\prob(O^{m,s,t})}\sum_{\tilde{\gamma}:(m,s,t)\textrm{-nice}} \prob\left((\gamma) = (\tilde{\gamma}) ; O^{m,s,t}\right)
\nonumber\\
& \ge & C_{16}^{m^{2}} \prod_{i,j} \prob(\Gamma_{i,j}^{s,t} | O_{i,j}^{s,t}),
\label{eq:theGammasAlmInd}
\end{eqnarray}
where the inequality follows from \eqref{eq:SplitTotal} and the obvious inequality
$\prob(O^{m,s,t}) \le \prod_{i,j} \prob(O_{i,j}^{s,t})$.
This gives part (i) of the lemma because
for each factor in the product of the last expression in \eqref{eq:theGammasAlmInd} we have,
by Definition \ref{def:nice}, Markov's inequality and Lemma \ref{lem:ExpSizeEpsxX},
\begin{eqnarray*}
 \prob(\Gamma_{i,j}^{s,t} | O_{i,j}^{s,t}) & = & \prob\left(\mathbb{E}[\tilde{\mathcal{C}}(A_{i,j}')\: | \; O_{i,j}^{s,t};
\gamma_{i,j} ] \le 2 C_{10} st \pi(t) \mid O_{i,j}^{s,t}\right)\\
 & \ge & 1 - \frac{\mathbb{E}[\tilde{\mathcal{C}}(A_{i,j}')\: | \; O_{i,j}^{s,t}]}{ 2 C_{10} st \pi(t)} \ge \frac{1}{2}.
\end{eqnarray*}

To prove part (ii) first note that
\begin{eqnarray*}
 \mathbb{E}\left[\tilde{\mathcal{C}}(A_{i,j}') | O^{m,s,t}; (\gamma) = (\tilde{\gamma})\right] & = &
\frac{\sum_{v \in A_{i,j}'} \prob\left(v \lra \partial A_{i,j}'; O^{m,s,t}; (\gamma) = (\tilde{\gamma})\right)}
{\prob(O^{m,s,t}; (\gamma) = (\tilde{\gamma}))}\\
 & \le & \frac{\sum_{v \in A_{i,j}'} \prob(v \lra \partial A_{i,j}'; O_{i,j}^{s,t}; \gamma_{i,j} = \tilde{\gamma}_{i,j})\prob(D)}{C_{16}\prob(D)\prob(O_{i,j}^{s,t}; \gamma_{i,j} = \tilde{\gamma}_{i,j})},
\end{eqnarray*}
where we used \eqref{eq:SplitOnxEpsWithGamma} in the denominator.
Hence
\begin{equation}\label{eq:expCAOleExpCAOij}
 \mathbb{E}\left[\tilde{\mathcal{C}}(A_{i,j}') | O^{m,s,t}; (\gamma) = (\tilde{\gamma})\right] \le
\frac{1}{C_{16}}\mathbb{E}[\tilde{\mathcal{C}}(A_{i,j}') | O_{i,j}^{s,t}; \gamma_{i,j} = \tilde{\gamma}_{i,j}]
\le \frac{2 C_{10}}{C_{16}} s t \pi(t),
\end{equation}
where the last inequality is just the `niceness' property (Definition \ref{def:nice}) of $(\tilde \ga)$.
%
To finish the proof of part (ii), note that, for each $K >0$,
\begin{eqnarray*} 
 \lefteqn{\prob\left(\mathcal{C}_{0,0}(Q) \le K m^2 s t \pi(t) | O^{m,s,t}; (\gamma) = (\tilde{\gamma})\right) \ge
1 - \frac{\mathbb{E}\left[\mathcal{C}_{0,0}(Q) | O^{m,s,t}; (\gamma) = (\tilde{\gamma})\right]}{K m^2 s t \pi(t)}}\\ 
 & \ge & 1 - \frac{\mathbb{E}\left[\tilde{\mathcal{C}}(\Lambda_{n}\setminus \Lambda_{ms})\right] +
\sum_{i,j} \mathbb{E}\left[\tilde{\mathcal{C}}(A_{i,j}') | O^{m,s,t}; (\gamma) = (\tilde{\gamma})\right]}{K m^2 s t \pi(t)}.
\end{eqnarray*}
Applying part (iv) of Theorem \ref{lem:PiBounds} to the first expectation in the r.h.s. of the last expression,
 and \eqref{eq:expCAOleExpCAOij} to each of the other expectations gives, by choosing $K$ sufficiently
large, the desired result. This completes the proof of part (ii) of Lemma \ref{lem:probGamma}.\qed
\end{proof}

\subsection{Cluster-size contributions inside the circuits}
In this section we write the value $t$ (the width of the relevant annuli and `corridors' in the construction) 
as $\lfloor \varepsilon s \rfloor$. A suitable value for $\varepsilon$ (depending on the values of $a$ and $b$ in
the statement of Theorem \ref{thm:maxClusterInInt}) will be determined in the next section. The main result in the
current section concerns the contribution from the interior of a nice circuit to
the cluster of that circuit.
Recall the notation \eqref{eq:defCijGamma}.
\begin{lemma}\label{lem:sizeIntGam}
There exist constants $C_{17}, C_{18}, \delta_{2} > 0$, and for every $\varepsilon < \frac{1}{12}$ there exists
$\delta_{3}(\varepsilon) > 0$, such that for all $s \in \mathbb{N}$ and all
$(s,\lfloor \varepsilon s\rfloor)$-nice circuits $\tilde{\gamma}_{0,0}$ in $A_{0,0}^{II}$,
\begin{enumerate}[(i)]
 \item \begin{equation} \label{eq-r-leminside-i}
 \prob\left(\mathcal{C}^{\tilde{\gamma}_{0,0}} \in
(C_{17}s^2 \pi(s), C_{18}s^2 \pi(s))\; | \; \gamma_{0,0} = \tilde{\gamma}_{0,0}\right) \ge \delta_{2}.
\end{equation}
\item \begin{equation} \label{eq-r-leminside-ii}
\prob\left(\mathcal{C}^{\tilde{\gamma}_{0,0}} <
4C_{10} s\lfloor \varepsilon s\rfloor \pi(\lfloor \varepsilon s\rfloor)\; | \;  \gamma_{0,0} = \tilde{\gamma}_{0,0}\right) \ge \delta_{3}(\varepsilon).
\end{equation}
\end{enumerate}
\end{lemma}
\begin{proof}
Let $B_{0,0}'=B_{0,0} \setminus A_{0,0}'$.
Let $\hat{Y} = |\{ v \in B_{0,0}': v \lra \partial B_{0,0}\}|$.
Clearly,
\begin{equation}\label{eq:compWithHatY}
 \prob(\mathcal{C}^{\tilde{\gamma}_{0,0}} \ge C_{17}s^2 \pi(s) \; | \; \gamma_{0,0} = \tilde{\gamma}_{0,0})
\ge \prob(\hat{Y} \ge C_{17}s^2 \pi(s)),
\end{equation}
which (for a suitable choice of $C_{17}$)  by Theorem \ref{lem:momentBoundExpConnBoundary}
is at least a positive constant, which we write as $2 \delta_{2}$.
To complete the proof we need to find a $C_{18}>0$ such that
\begin{equation}\label{eq:probLargeClusIsSmall}
 \prob(\mathcal{C}^{\tilde{\gamma}_{0,0}} \ge C_{18}s^2 \pi(s) \; | \; \gamma_{0,0} = \tilde{\gamma}_{0,0}) \le \delta_{2}.
\end{equation}
To do this we look for an upper bound for
$\mathbb{E}[\mathcal{C}^{\tilde{\gamma}_{0,0}} \; | \; \gamma_{0,0} = \tilde{\gamma}_{0,0}]$.
We have
\begin{eqnarray} \label{eq-r-ub}
 \mathbb{E}[\mathcal{C}^{\tilde{\gamma}_{0,0}} \; | \; \gamma_{0,0} =
\tilde{\gamma}_{0,0}] & = & \mathbb{E}[\mathcal{C}^{\tilde{\gamma}_{0,0}} \; |
\; O_{0,0}^{s,\lfloor \varepsilon s\rfloor}; \gamma_{0,0} = \tilde{\gamma}_{0,0}] \\
 & \le & \mathbb{E}[\tilde{\mathcal{C}}(B_{0,0}')]
 + \mathbb{E}[\tilde{\mathcal{C}}(A_{0,0}')\; | \; O_{0,0}^{s,\lfloor \varepsilon s\rfloor};
\gamma_{0,0} = \tilde{\gamma}_{0,0}] \nonumber
\end{eqnarray}
Applying part (iv) of Theorem \ref{lem:PiBounds} to the first expectation in the last line, and the niceness property of
$\tilde{\gamma}_{0,0}$ to the other expectation, shows that
the l.h.s. of \eqref{eq-r-ub}
is at most $C_{19} s^2 \pi(s)$. Finally, Markov's inequality
gives \eqref{eq:probLargeClusIsSmall} with $C_{18} = \frac{C_{19}}{\delta_{2}}$.
This completes the proof of part (i).

Now we prove part (ii). Let $G$ be the event that there is a closed dual circuit in $A_{0,0}^{III}$.
On this event, let $\beta_{0,0}$ denote the innermost of such circuits. Observe that, conditioned on $\be_{0,0}$,
the configuration outside $\beta_{0,0}$
is independent of the configuration inside. Also observe that, on the event $G$, all vertices in the interior
of $\ga_{0,0}$ that are connected to $\ga_{0,0}$ are in $A'_{0,0}$.
By these and related simple observations we have that
\begin{eqnarray*} 
 \lefteqn{\prob\left(\mathcal{C}^{\tilde{\gamma}_{0,0}} < 4C_{10} s\lfloor \varepsilon s\rfloor \pi(\lfloor \varepsilon s\rfloor)\; |
\;  \gamma_{0,0} = \tilde{\gamma}_{0,0}; \beta_{0,0} = \tilde{\beta}\right)}\\
& = &  \prob\left(\mathcal{C}^{\tilde{\gamma}_{0,0}} <
4C_{10} s\lfloor \varepsilon s\rfloor \pi(\lfloor \varepsilon s\rfloor)\; |
\;  O_{0,0}^{s,\lfloor \varepsilon s\rfloor}; \gamma_{0,0} = \tilde{\gamma}_{0,0}; \beta_{0,0} = \tilde{\beta}\right) \\
 & \ge & \prob\left(\tilde{\mathcal{C}}(A_{0,0}') < 4C_{10} s\lfloor \varepsilon s\rfloor \pi(\lfloor \varepsilon s\rfloor)\; |
\;  O_{0,0}^{s,\lfloor \varepsilon s\rfloor}; \gamma_{0,0} = \tilde{\gamma}_{0,0}\right),
\end{eqnarray*}
which, by Markov's inequality and because $\tilde\ga_{0,0}$ is nice, is at least $1/2$.
Hence, the l.h.s. of \eqref{eq-r-leminside-ii} is at least $(1/2) \prob(G)$, which by RSW is larger than
some positive constant which depends only on $\varepsilon$.

\qed
\end{proof}

\subsection{Completion of the proof of Theorem \ref{thm:maxClusterInInt}} \label{sect-compl}
We are now ready to prove Theorem \ref{thm:maxClusterInInt}. First we still restrict to the case $p = 1/2$.
Let $0 < a < b$ be given.
See the brief outline in Section \ref{sec:proofDefi}.
The lengths of the building blocks $B_{i,j}$ and the widths of the annuli and `corridors'
in the partition of $\La_n$, will be taken proportional to $n$, say (roughly) $x n$ and $\ep x n$ respectively,
with suitably chosen $x$ and $\ep$. For this purpose we will use the following lemma:

\begin{lemma}\label{lem:chooseNxEps}
There exist $x>0, \varepsilon \in (0,\frac{1}{12})$ and
$N \in \mathbb{N}$,
with $\frac{1}{x}$ an odd integer, such that for all $n \ge N$ the following inequalities hold:
\begin{eqnarray}
 C_{17}\xn^2 \pi(\xn)\cdot (\frac{1}{x})^{2} & \ge & an^{2}\pi(n); \label{eq:xCondI}\\
 C_{18}\xn^{2} \pi(\xn) & \le & \frac{1}{3}(b-a)n^{2}\pi(n); \label{eq:xCondII} \\
 (4C_{10} \vee C_{15}) (\frac{1}{x})^{2} \exn\xn\pi(\exn) & \le & \frac{1}{3}(b-a)n^{2}\pi(n).\label{eq:EpsCond}
\end{eqnarray}
\end{lemma}
\begin{proof}
It is easy to see (a weak form of the lower bound in part (i) of Theorem \ref{lem:PiBounds} suffices)
that if $x$ is sufficiently
(depending on $a$) small, then
\eqref{eq:xCondI} holds for all sufficiently large $n$.
It also  easily follows (now from the upper bound in the same Theorem) that if $x$ is sufficiently (depending on
$b-a$) small, \eqref{eq:xCondII} holds for all sufficiently large $n$.
Finally, for $x$ fixed, it follows (again from the upper bound in part (i) of Theorem \ref{lem:PiBounds}) that
if $\ep$ is sufficiently (depending on $b-a$ and $x$) small, then \eqref{eq:EpsCond} holds for all sufficiently large $n$. \\
{\em Note that for this last step it is essential that the exponent ($1/2$) in part (i) of
Theorem \ref{lem:PiBounds} is strictly smaller than $1$.} \\
This completes the proof of Lemma \ref{lem:chooseNxEps}. \qed
\end{proof}

\medskip
Now let $x$, $\ep$ and $N$ be as in Lemma \ref{lem:chooseNxEps}. Moreover we assume
(which we may, because we can enlarge $N$ if necessary) that
$n - \frac{1}{x} \lfloor x n \rfloor \leq \lfloor \ep \lfloor x n \rfloor \rfloor$ for all $n \ge N$.
Denote by $D_{n}$ the event
\begin{equation*}
 D_{n} = \{ \exists v \in \Lambda_{n}: \mathcal{C}_{n}(v) \in (an^2 \pi(n), bn^2 \pi(n)) \}.
\end{equation*}
Let $n \ge N$ and let $m = \frac{1}{x}, s = \xn$ and $t = \exn$.
By straightforward RSW and FKG arguments, there is a $\delta_{4}(x,\varepsilon)>0$ such that
$\prob(O^{m,s,t}) > \delta_{4}(x,\varepsilon)$.
Hence
\begin{equation}\label{eq:r-end1}
\prob(D_{n}) \ge \delta_{4}(x,\varepsilon)\prob(D_{n} | O^{m,s,t}).
\end{equation}
From Lemma \ref{lem:probGamma} (i) it follows that
\begin{equation} \label{eq:r-end2}
\prob(D_{n} | O^{m,s,t}) \ge (C_{14})^{\frac{1}{x^{2}}}\prob(D_{n} | O^{m,s,t}; \Gamma^{m,s,t}).
\end{equation}
The next step is conditioning on the widest open circuits.
\begin{equation} \label{eq:final-added}
 \prob(D_{n} | O^{m,s,t}; \Gamma^{m,s,t}) = \sum_{(\tilde{\gamma}): (m,s,t)\textrm{-nice}}
\prob(D_{n} | (\gamma) = (\tilde{\gamma}); O^{m,s,t}) \, \prob((\gamma) = (\tilde{\gamma}) | O^{m,s,t}; \Gamma^{m,s,t}).
\end{equation}
For each $(\tilde{\gamma})$ we denote by $\mathcal{C}_{(\tilde{\gamma})}^{in}$ the number of all vertices that are in the
interior of a circuit in the collection $(\tilde{\gamma})$ and connected to that circuit,
and by $\mathcal{C}_{(\tilde{\gamma})}^{out}$ the number of vertices outside these circuits
that are connected to one or more of these circuits, plus the number of vertices on these circuits. We have
\begin{eqnarray} \label{eq:r-end3}
\lefteqn{\prob(D_{n} | (\gamma) = (\tilde{\gamma}); O^{m,s,t})} \\ \nonumber
 & \ge & \prob\left(\frac{\mathcal{C}_{(\tilde{\gamma})}^{in}}{n^{2}\pi(n)} \in (a,b - \frac{1}{3}(b-a)) \; | \; \frac{\mathcal{C}_{(\tilde{\gamma})}^{out}}{n^{2}\pi(n)} \le \frac{1}{3}(b-a); (\gamma) = (\tilde{\gamma}); O^{m,s,t}\right) \\ \nonumber
 & & \cdot \prob\left(\mathcal{C}_{(\tilde{\gamma})}^{out} \le \frac{1}{3}(b-a)n^{2}\pi(n) \; | \; (\gamma) = (\tilde{\gamma});
O^{m,s,t}\right) \\ \nonumber
 & \geq & \frac{1}{2} \prob\left(\frac{\mathcal{C}_{(\tilde{\gamma})}^{in}}{n^{2}\pi(n)} \in
(a,\frac{1}{3} a + \frac{2}{3} b)\right),
\end{eqnarray}
where the last inequality holds by \eqref{eq:EpsCond} and Lemma \ref{lem:probGamma} (ii), and because the
configurations in the interiors of the $\tilde{\gamma}_{i,j}$'s 
are obviously independent of  the event conditioned on in the expression in the r.h.s. of the first inequality.
Note that $\mathcal{C}_{(\tilde{\gamma})}^{in} = \sum_{i,j} \mathcal{C}^{\tilde{\gamma}_{i,j}}$.
The $\mathcal{C}^{\tilde{\gamma}_{i,j}}$'s are independent
and for each $i,j$ we have
\begin{eqnarray*}
 \lefteqn{\prob\left(\mathcal{C}^{\tilde{\gamma}_{i,j}} \in (ax^{2} n^{2}\pi(n), \frac{1}{3}(b-a)n^{2}\pi(n))\right)}\\
 & \stackrel{\eqref{eq:xCondI}, \eqref{eq:xCondII}}{\ge} & \prob\left(\mathcal{C}^{\tilde{\gamma}_{i,j}} \in
\left(C_{17}\xn^{2}\pi(\xn), C_{18}\xn^{2}\pi(\xn)\right)\right) \stackrel{\textrm{Lem.} \ref{lem:sizeIntGam} (i)}{\ge}
\delta_{2},
\end{eqnarray*}
and
\begin{eqnarray*}
 \lefteqn{\prob\left(\mathcal{C}^{\tilde{\gamma}_{i,j}} \le x^{2}\frac{1}{3}(b-a)n^{2}\pi(n)\right)} \\
 & \stackrel{\eqref{eq:EpsCond}}{\ge} & \prob\left(\mathcal{C}^{\tilde{\gamma}_{i,j}}
\le 4C_{10}\exn\xn\pi(\exn)\right) \stackrel{\textrm{Lem.} \ref{lem:sizeIntGam} (ii)}{\ge} \delta_{3}(\varepsilon).
\end{eqnarray*}
Hence
the conditions of Lemma \ref{lem:sumOfRandomVars} (with $k = (1/x)^2$, $\al = a n^2 \pi(n)$, $\be = (\frac{1}{3} a + 
\frac{2}{3} b) n^2 \pi(n)$) are satisfied. Hence, by that lemma the l.h.s. of \eqref{eq:r-end3} is at least
$\frac{1}{2} (\delta_{2} \wedge \delta_{3}(\ep))^{(1/x)^2}$. Together with \eqref{eq:r-end1} - \eqref{eq:final-added}
this shows that 
\begin{equation}\label{eq:bound-Dn}
\P(D_n) > \de_5,
\end{equation}
with $\de_5$ a positive constant which depends only on $a$ and $b$.

Now we will show that, by the way we `constructed' the open cluster, a similar result holds for the
{\em maximal} open cluster in $\La_n$.
First note that the `constructed' cluster has the property that it contains an open horizontal and
an open vertical crossing of the box $\La_{m s}$. Also note that there is at most one open cluster
with this property. Given the exact location of the (unique) open cluster with this property, the conditional 
probability that it is the maximal cluster in $\La_n$ is, if its size is larger than $a n^2 \pi(n)$, clearly
larger than or equal to the probability that the remaining part of $\La_n$ contains no open cluster of size
larger than $a n^2 \pi(n)$. By obvious monotonicity this probability is at least $\P(\mathcal{M}_n \leq a n^2 \pi(n))$,
which by Theorem \ref{thm:maxClustCanBeSmall} is at least some positive constant $\de_2$ (which depends only on $a$).
This argument gives
$$\P\left(\mathcal{M}_n \in (a n^2 \pi(n), b n^2 \pi(n))\right) \geq \de_2 \de_5,$$
which completes the proof of Theorem \ref{thm:maxClusterInInt} for $p = 1/2$.

Now let, more generally, $p$ be such that $L(p) \geq n$. It is straightforward to check that (due to the `uniformity'
in $p$ of the results in Section \ref{sect-notation}) each step in the proof
remains essentially valid. For instance, it is easy to see from the arguments used that Lemma \ref{lem:ExpSizeEpsxX} 
(now with $\P$ replaced by $\P_p$) remains valid as long as $L(p) \geq s$. Since we take $s \leq n$ (application of)
this lemma (and, similarly, the other lemma's) can be carried out as before.
This completes the proof of Theorem \ref{thm:maxClusterInInt}. \qed

\medskip\noindent
{\bf \large Acknowledgment.} The first author thanks Antal J\'arai for a useful and pleasant discussion on these an related
problems.

\end{document}